\documentclass[11pt]{article}
\PassOptionsToPackage{obeyspaces}{url}
\usepackage[hidelinks]{hyperref}
\usepackage{bookmark}
\bookmarksetup{open,numbered,depth=5}
\usepackage{amsfonts}
\usepackage{amsmath}
\usepackage{amssymb, amsthm, enumitem}
\usepackage{mathrsfs}
\usepackage{pgf,tikz}
\usepackage{bm}
\usepackage{array}
\usepackage{abraces}

\usetikzlibrary{calc}

\usepackage{etoolbox} 
\patchcmd{\thebibliography}{\leftmargin\labelwidth}{\leftmargin\labelwidth\addtolength\itemsep{-0.2\baselineskip}}{}{}

\usepackage{mathtools}
\usepackage{xstring}

\author{
Zichao Dong\thanks{Extremal Combinatorics and Probability Group (ECOPRO), Institute for Basic Science (IBS), Daejeon, South Korea. Supported by the Institute for Basic Science (IBS-R029-C4). \texttt{zichao@ibs.re.kr}. }
\and Minghui Ouyang\thanks{School of Mathematical Sciences, Peking University, Beijing 100871, China. \texttt{ouyangminghui1998@gmail.com}.}
\and Lan Wei\thanks{Mathematical Research Center, Shandong University, Jinan 250100, China \texttt{lanwei@mail.sdu.edu.cn}.}
}

\title{Sharp bounds on $k$-wise generalizations of oddtowns and eventowns}
\date{}

\usepackage[nameinlink]{cleveref}

\newtheorem{theorem}{Theorem}
\newtheorem{lemma}[theorem]{Lemma}
\newtheorem{corollary}[theorem]{Corollary}

\newtheorem{proposition}[theorem]{Proposition}

\newtheorem{problem}[theorem]{Problem}

\newcommand*{\eqdef}{\stackrel{\mbox{\normalfont\tiny def}}{=}} 

\newcommand*{\N}{\mathbb{N}}                                    
\newcommand*{\Z}{\mathbb{Z}}                                    
\newcommand*{\R}{\mathbb{R}}                                    
\newcommand*{\F}{\mathbb{F}}
\renewcommand*{\=}{\overset{\F_2}{=}}

\renewcommand*{\sc}{\mathsf{c}}
\newcommand*{\cF}{\mathcal{F}}
\newcommand*{\cG}{\mathcal{G}}
\newcommand*{\cA}{\mathcal{A}}
\newcommand*{\cB}{\mathcal{B}}
\newcommand*{\valpha}{{\bm{\alpha}}}
\newcommand*{\vbeta}{{\bm{\beta}}}
\newcommand*{\vgamma}{{\bm{\gamma}}}

\newcommand*{\vI}{{\bm{I}}}
\newcommand*{\mGamma}{{\bm{\Gamma}}}
\newcommand*{\mM}{{\bm{M}}}
\newcommand*{\wF}{{\widetilde{\F}}}

\DeclareMathOperator*{\lv}{lv}
\DeclareMathOperator*{\grd}{grd}
\DeclareMathOperator*{\spa}{span}

\allowdisplaybreaks[4]

\begin{document}

\maketitle

\begin{abstract}
    For $\valpha = (\alpha_1, \dots, \alpha_k) \in \F_2^k$, an $\valpha$-\emph{town} is a set family in which every $i$-wise intersection has parity $\alpha_i$. Denote by $f_{\valpha}(n)$ the maximum size of an $\valpha$-town on $[n]$. The classical oddtown and eventown problems study the cases $\valpha = (1, 0)$ and $(0, 0)$, respectively. We determine the sharp asymptotics of $f_{\valpha}(n)$ for all $\valpha$, answering questions of Johnston--O'Neill and Wei--Zhang--Ge. 
    
    We also study a symmetric variant $g_{\valpha}(n)$, in which $i$-wise intersection sizes $|F_1 \cap \dots \cap F_i|$ are replaced by $i$-wise intersection-union sizes $|F_1 \cap \dots \cap F_i| + |F_1 \cup \dots \cup F_i|$. 
\end{abstract}

\section{Introduction} \label{sec:intro}


Restricted-intersection problems have long been a central topic in the study of extremal set theory and have found numerous striking applications. For each $n \in \N_+ \eqdef \{1, 2, \dots\}$, write $[n] \eqdef \{1, \dots, n\}$ and let $2^{[n]}$ denote the power set of $[n]$. As a classical example of a restricted-intersection result: 

\vspace{-0.25em}
\begin{center}
\begin{minipage}{0.9\textwidth}
    \emph{Suppose $p \in \N_+$ is a prime, and let $\cF \subseteq 2^{[n]}$ be a family of $(2p-1)$-subsets of $[n]$. If no two distinct sets in $\cF$ intersect in exactly $p-1$ elements, then $|\cF| \le \binom{n}{0} + \binom{n}{1} + \dots + \binom{n}{p}$.} 
\end{minipage}
\end{center}
\vspace{-0.25em}

\noindent This is a special case of the celebrated Frankl--Wilson theorem~\cite{1981F-W}. Remarkably, it yields exponentially large explicit constructions for diagonal Ramsey numbers; see~\cite{matousek} for an elegant exposition. Restricted-intersection results, including~\cite{deza_erdos_frankl,1981F-W,babai_frankl_kutin_stefankovic}, apply to distance-related coding theory problems; see~\cite{1995Babaicode,2002Sudakovdistance,dong_gao_liu_ouyang_zhou}. We also refer to~\cite{2018Franklsurvey} for a comprehensive survey of restricted-intersection problems. 


The oddtown and eventown problems provide two famous examples of restricted-intersections in which the constraints are imposed on the parities of set sizes and pairwise intersection sizes. Call $\cF \subseteq 2^{[n]}$ an \emph{eventown} (resp.~\emph{oddtown}) if every set in $\cF$ has even (resp.~odd) size, and the intersection of any two distinct sets in $\cF$ has even size. Berlekamp \cite{1969Berlekamp} and Graver \cite{1975Graver} obtained sharp extremal bounds: every oddtown contains at most $n$ sets, while every eventown contains at most $2^{\lfloor n/2\rfloor}$ sets. The proofs are early paradigms of \emph{linear algebraic methods} in combinatorics. Recently, the oddtown problem has also found unexpected applications to multicolor Ramsey constructions~\cite{2021Conlon-FeberRamsey,2021Wigderson}.



Consider a natural analogue---$\ell$-\emph{oddtown}: a family $\cF \subseteq 2^{[n]}$ such that $|A| \not\equiv 0 \pmod \ell$ for any $A \in \cF$ and $|A \cap B| \equiv 0 \pmod \ell$ for any distinct $A, B \in \cF$. For prime $\ell$, and more generally prime-power $\ell$, the sharp bound $|\cF| \le n$ is known, but the case of composite $\ell$ is much less understood. Szegedy~\cite[Ex.~1.1.28]{1992babai-frankl} has established the first improvement over the trivial bounds for composite moduli, and the recent work of Bukh, Chao, and Zheng~\cite{2025Bukhcomposite} gives further improved bounds. 

The study of $\ell$-\emph{divisible} families, where pairwise or $i$-wise intersection sizes are restricted to be divisible by $\ell$, has also attracted considerable attention; see  \cite{deza_erdos_frankl,1983Frankl-Odlyzko,2005Szabo-vu,2022Gishboliner-Sudakov-Tomon}. Many other extensions of oddtowns and eventowns have been developed over the years~\cite{1983Deza5,1997Vu_odd,1999Vu,2004Furedi-Sudakov,2023ONeill,2026Dey}. 

\medskip

In this paper, we study generalizations of oddtowns and eventowns in which every $i$-wise intersection size in a set family has restricted parity, as introduced by Johnston and O'Neill~\cite{johnston_oneill}. 
Several previous partial results were obtained, see \cite{oneill_verstraete,2018Sudakov_Vieira}. 
Recall standard asymptotic notations. Let $f, g \colon \N_+ \to \R_{\ge 0}$ be functions with finitely many zeroes. Write
\vspace{-0.5em}
\begin{itemize}
    \item $f = o(g)$ or $g = \omega(f)$ if $\lim\limits_{n \to \infty} \frac{f(n)}{g(n)} = 0$, and $f \sim g$ if $\lim\limits_{n \to \infty} \frac{f(n)}{g(n)} = 1$; 
    \vspace{-0.5em}
    \item $f = O(g)$ or $g = \Omega(f)$ if there exists $c > 0$ such that $f(n) \le cg(n) \, (\forall n \in \N_+)$; 
    \vspace{-0.5em}
    \item $f = \Theta(g)$ if both $f = O(g)$ and $g = O(f)$ hold. 
\end{itemize}
\vspace{-0.5em}

\paragraph{The \texorpdfstring{$\valpha$}{alpha}-town problem.} For $\valpha = (\alpha_1, \dots, \alpha_k) \in \F_2^k$, we refer to a set family $\cF$ as an $\valpha$-\emph{town} if, for $i = 1, \dots, k$, the equality $|F_1 \cap \dots \cap F_i| \equiv \alpha_i \pmod 2$ holds for any distinct $F_1, \dots, F_i \in \cF$. For instance, an oddtown is a $(1, 0)$-town, and an eventown is a $(0, 0)$-town.

For $\cF \subseteq 2^{[n]}$, define $f_{\valpha}(n)$ as the maximum size of $|\cF|$, provided that $\cF$ is an $\valpha$-town. Then, the oddtown theorem and eventown theorem assert $f_{(1, 0)}(n) = n$ and $f_{(0, 0)}(n) = 2^{\lfloor n/2 \rfloor}$, respectively. 

In their pioneering study of $f_{\valpha}(n)$, Johnston and O'Neill~\cite{johnston_oneill} found the exact values of $f_{\valpha}(n)$ for all $\valpha \in \F_2^2 \cup \F_2^3$, and obtained sharp asymptotics for each $\valpha \in \F_2^4$. Improving an O'Neill--Verstra\"{e}te result~\cite{oneill_verstraete}, Wei, Zhang, and Ge~\cite{wei_zhang_ge} proved $f_{\valpha}(n) \sim (t!n)^{1/t}$ for every $\valpha \in \F_2^k$ whose last $1$-coordinate appears before the $\frac{k}{2}$-th bit. In the same paper, they also showed that, for every $c > 0$, the number of $\valpha \in \F_2^k$ satisfying $f_{\valpha}(n) = \Omega(n^c)$ is bounded by a constant independent of $k$. 

Wei, Zhang, and Ge~\cite[Problem VI.2]{wei_zhang_ge} posed the following problem. 

\begin{problem} \label{prob:main}
    Determine the sharp asymptotic behavior of $f_\valpha(n)$ for every $k \in \N_+$ and every $\valpha \in \F_2^k$. 
\end{problem}

As an intermediate step toward \Cref{prob:main}, Johnston and O'Neill~\cite[Question 1]{johnston_oneill} asked whether there exists an absolute constant $C$ such that, for every $k \in \N_+$, there are at most $C$ vectors $\valpha \in \F_2^k$ satisfying $f_{\valpha}(n) = \Omega(\sqrt{n})$. Wei, Zhang, and Ge~\cite[Problem VI.1]{wei_zhang_ge} answered this question affirmatively by showing that one may take $C = 2^7$, and further asked for the optimal value of $C$. 

Our main result (\Cref{thm:main}) resolves \Cref{prob:main}. The full statement of the theorem is somewhat technical, however, and requires some linear-algebraic preliminaries. In particular, we shall see that $f_{\valpha}(n) = \Omega(\sqrt{n})$ if and only if $\lv(\valpha) \le 2$, from which it follows that the optimal value of $C$ is $2^5$. 


\medskip

Write $\F_2^{\otimes \omega} \eqdef \prod_{j=1}^{\infty} \F_2$ as the countable direct product of copies of $\F_2$. That is, the vector space of all infinite binary sequences. For every such vector $\vgamma \in \F_2^{\otimes \omega}$, we denote
\vspace{-0.5em}
\begin{itemize}
	\item by $(\vgamma)_j$ (or $\vgamma_j$) the $j$-th coordinate of the vector $\vgamma$, 
    \vspace{-0.5em}
	\item by $(\vgamma)_{[k]}$ (or $\vgamma_{[k]}$) the vector $(\vgamma_1, \dots, \vgamma_k) \in \F_2^k \, (\forall k)$, and
    \vspace{-0.5em}
	\item by $\sigma(\vgamma) \in \F_2^{\otimes \omega}$ in which $\sigma(\vgamma)_j \eqdef \vgamma_{j+1} \, (\forall j)$ the left shift of $\vgamma$. 
\end{itemize}
\vspace{-0.5em}
Introduce an infinite sequence of \emph{canonical vectors} $\vgamma^1, \vgamma^2, \dots \in \F_2^{\otimes \omega}$, where for every $i, j$ we have
\[
(\vgamma^{2i-1})_j \eqdef \binom{-i}{j-1} \bmod 2 = \binom{i+j-2}{j-1} \bmod 2, \qquad (\vgamma^{2i})_j \eqdef \binom{i-1}{j-1} \bmod 2. 
\]
Notice that we use the extended binomial coefficients. That is, for any $a \in \Z$ and $b \in \Z_{\ge 0}$, 
\[
\binom{-a}{b} \eqdef \frac{(-a)(-a-1)\dots(-a-b+1)}{b!}. 
\]
The first canonical vector $\vgamma^1$ is the infinite all-$1$ vector, and the second through the seventh are
\begin{align*}
    \vgamma^2 = (1, 0, 0, 0, 0, 0, 0, 0, 0, 0, 0, 0, \dots), &\qquad \vgamma^3 = (\aunderbrace[l1r]{1, 0}, \aunderbrace[l1r]{1, 0}, \aunderbrace[l1r]{1, 0}, \aunderbrace[l1r]{1, 0}, \aunderbrace[l1r]{1, 0}, \aunderbrace[l1r]{1, 0}, \dots), \\
    \vgamma^4 = (1, 1, 0, 0, 0, 0, 0, 0, 0, 0, 0, 0, \dots), &\qquad \vgamma^5 = (\aunderbrace[l1r]{1, 1, 0, 0}, \aunderbrace[l1r]{1, 1, 0, 0}, \aunderbrace[l1r]{1, 1, 0, 0}, \dots), \\
    \vgamma^6 = (1, 0, 1, 0, 0, 0, 0, 0, 0, 0, 0, 0, \dots), &\qquad \vgamma^7 = (\aunderbrace[l1r]{1, 0, 0, 0}, \aunderbrace[l1r]{1, 0, 0, 0}, \aunderbrace[l1r]{1, 0, 0, 0}, \dots). 
\end{align*}

\medskip

Informally, the asymptotics of $f_{\valpha}(n)$ is determined by the decomposition of $\valpha$ into the canonical vectors. Before giving a precise formulation, we require the following linear independence result. 

\begin{lemma} \label{lem:canonical}
    For every $k \in \N_+$, the vectors $(\vgamma^1)_{[k]}, \dots, (\vgamma^k)_{[k]}$ form a basis of $\F_2^k$. 
\end{lemma}

Due to \Cref{lem:canonical}, for every vector $\valpha \in \F_2^k$ there are unique coefficients $c_1, \dots, c_k \in \F_2$ such that
\begin{equation} \label{eq:canonical}
    \valpha = c_1 \cdot (\vgamma^1)_{[k]} + \dots + c_k \cdot (\vgamma^k)_{[k]}. 
\end{equation}
We refer to this representation as the \emph{canonical decomposition} of $\valpha$. 

We need to introduce the \emph{level} and \emph{grade} of a vector $\valpha \in \F_2^k$ with canonical decomposition \eqref{eq:canonical}. Firstly, for each canonical vector $\vgamma^i$, we define its \emph{level} as $\lv(\vgamma^i) \eqdef \lfloor i/2 \rfloor$. Secondly, the \emph{level} of $\valpha$ corresponds to that of the highest canonical vector appearing in its canonical decomposition. That is, $\lv(\valpha) \eqdef \max \bigl\{ \lv(\vgamma^i) : c_i = 1 \bigr\}$. The all-$0$ vector of arbitrary length has level $0$. Note that $\lv$ is an operator from $\bigsqcup\limits_{k = 1}^{\infty} \F_2^k$ to non-negative integers $\Z_{\ge 0}$. If $\lv(\valpha) = t > 0$, then we define its \emph{grade} as
\vspace{-0.75em}
\[
\grd(\valpha) \eqdef \begin{cases}
    1 \qquad &\text{if one of $\vgamma^{2t}$ and $\vgamma^{2t+1}$ appears in \eqref{eq:canonical}, equivalently $c_{2t} \cdot c_{2t+1} = 0$}, \\
    2 \qquad &\text{if both $\vgamma^{2t}$ and $\vgamma^{2t+1}$ appear in \eqref{eq:canonical}, equivalently $c_{2t} \cdot c_{2t+1} = 1$}.
\end{cases}
\]

Our main result below can be informally summarized as ``the highest-level principle'': the theorem asserts that the sharp asymptotics of $f_{\valpha}(n)$ is determined by the level and the grade of $\valpha$. 

\begin{theorem} \label{thm:main}
    Suppose $k$ is a positive integer and let $\valpha \in \F_2^k$. As $n \to \infty$ the following holds. 
    \[
    f_{\valpha}(n) = \begin{cases}
        2^{\lfloor n/2 \rfloor} \qquad &\text{if $\lv(\valpha) = 0$ and $\valpha = (0, \dots, 0)$}, \\
        2^{\lfloor (n-1)/2 \rfloor} \qquad &\text{if $\lv(\valpha) = 0$ and $\valpha = (1, \dots, 1)$}, \\
        \bigl( 1 + o(1) \bigr) \cdot (t!)^{1/t} \cdot n^{1/t} \qquad &\text{if $\lv(\valpha) = t > 0$ and $\grd(\valpha) = 1$}, \\
        \bigl( 1 + o(1) \bigr) \cdot (t!/2)^{1/t} \cdot n^{1/t} \qquad &\text{if $\lv(\valpha) = t > 0$ and $\grd(\valpha) = 2$}. 
    \end{cases}
    \]
\end{theorem}

As an illustration, the table below exhibits for every $\valpha \in \F_2^3$ its canonical decomposition, level, grade, together with the known value of $f_{\valpha}(n)$ due to Johnston and O'Neill~\cite[Table 1]{johnston_oneill}. 

\smallskip

\begin{center}
    \begin{tabular}{|c|c|c|c|c|}
        \hline
        $\valpha$ & canonical decomposition & lv & grd & $f_{\valpha}(n)$ \\
        \hline
        $(0, 0, 0)$ & $0 \cdot (\vgamma^1)_{[3]} + 0 \cdot (\vgamma^2)_{[3]} + 0 \cdot (\vgamma^3)_{[3]}$ & 0 & \slash & $2^{\lfloor n/2 \rfloor}$ \\
        \hline
        $(1, 1, 1)$ & $1 \cdot (\vgamma^1)_{[3]} + 0 \cdot (\vgamma^2)_{[3]} + 0 \cdot (\vgamma^3)_{[3]}$ & 0 & \slash & $2^{\lfloor (n-1)/2 \rfloor}$ \\
        \hline
        $(1, 0, 0)$ & $0 \cdot (\vgamma^1)_{[3]} + 1 \cdot (\vgamma^2)_{[3]} + 0 \cdot (\vgamma^3)_{[3]}$ & 1 & 1 & $n$ \\
        \hline
        $(0, 1, 1)$ & $1 \cdot (\vgamma^1)_{[3]} + 1 \cdot (\vgamma^2)_{[3]} + 0 \cdot (\vgamma^3)_{[3]}$ & 1 & 1 & $n - 1$ \\
        \hline
        $(1, 0, 1)$ & $0 \cdot (\vgamma^1)_{[3]} + 0 \cdot (\vgamma^2)_{[3]} + 1 \cdot (\vgamma^3)_{[3]}$ & 1 & 1 & $n$ or $n - 1$ \\
        \hline
        $(0, 1, 0)$ & $1 \cdot (\vgamma^1)_{[3]} + 0 \cdot (\vgamma^2)_{[3]} + 1 \cdot (\vgamma^3)_{[3]}$ & 1 & 1 & $n$ or $n - 1$ \\
        \hline
        $(0, 0, 1)$ & $0 \cdot (\vgamma^1)_{[3]} + 1 \cdot (\vgamma^2)_{[3]} + 1 \cdot (\vgamma^3)_{[3]}$ & 1 & 2 & $\lfloor n/2 \rfloor$ or $\lfloor n/2 \rfloor + 1$ \\
        \hline
        $(1, 1, 0)$ & $1 \cdot (\vgamma^1)_{[3]} + 1 \cdot (\vgamma^2)_{[3]} + 1 \cdot (\vgamma^3)_{[3]}$ & 1 & 2 & $\lfloor n/2 \rfloor$ or $\lfloor n/2 \rfloor + 1$ \\
        \hline
    \end{tabular}
\end{center}

\smallskip

Notice that the fundamental works of Berlekamp~\cite{1969Berlekamp} and Graver~\cite{1975Graver} already determine the maximum size of eventowns on $[n]$ for any $n \in \N_+$, so there is nothing to prove in the case $\lv(\valpha) = 0$. Henceforth, we shall just focus on the case $\lv(\valpha)>0$ in \Cref{thm:main}. 

\paragraph{The \texorpdfstring{$\valpha$}{alpha}-ville problem.} The $\valpha$-town condition is inherently ``asymmetric'', in the sense that it imposes parity constraints on intersection sizes but not on their complements. A natural variant is to impose parity constraints on ``agreements'' rather than ``intersections''. 

Specifically, we identify a set family $\cF \subseteq 2^{[n]}$ with a collection of $\{0, 1\}$-vectors $\cF \subseteq \{0, 1\}^n$, where each set is identified with its characteristic vector. For example, the set $\{1, 3, 6\} \in 2^{[7]}$ corresponds to the vector $(1, 0, 1, 0, 0, 1, 0) \in \{0, 1\}^7$. In this language, ``intersection'' counts the number of common $1$'s, while ``agreement'' counts the number of coordinates on which all vectors agree (i.e., the total number of common $0$'s and common $1$'s). For $F_1, \dots, F_m \in 2^{[n]}$, their agreement has cardinality
\[
\Bigl| \bigl( F_1 \cap \dots \cap F_m \bigr) \cup \bigl( [n] \setminus (F_1 \cup \dots \cup F_m) \bigr) \Bigr| \equiv n + |F_1 \cap \dots \cap F_m| + |F_1 \cup \dots \cup F_m| \pmod 2. 
\]
Hence, for any fixed $n$, prescribing the parity of agreement is equivalent to prescribing the parity of intersection-union $|F_1 \cap \dots \cap F_m| + |F_1 \cup \dots \cup F_m|$. This leads to the study of ``ville''. 

For $\valpha = (\alpha_1, \dots, \alpha_k) \in \F_2^k$, a set family $\cF$ is an $\valpha$-\emph{ville} if for $i = 1, \dots, k$, the modular equality
\[
|F_1 \cap \dots \cap F_i| + |F_1 \cup \dots \cup F_i| \equiv \alpha_i \pmod 2
\]
holds for any distinct $F_1, \dots, F_i \in \cF$. Denote by $g_{\valpha}(n)$ the maximum size of an $\valpha$-ville $\cF \subseteq 2^{[n]}$. 


Does an analogue of ``the highest-level principle'' (see Theorem~\ref{thm:main}) hold for $\valpha$-ville? We do not know how to establish such a statement. However, in contrast to the polynomial growth of $f_{\valpha}(n)$ in $n$ shown by \Cref{thm:main}, we can show that $g_{\valpha}(n)$ is ``often'' bounded. To be precise, we introduce the extended set $\wF_2 \eqdef \F_2 \cup \{*\}$, where $|F_1 \cap \dots \cap F_i| + |F_1 \cup \dots \cup F_i| \equiv * \pmod 2$ indicates that no restriction is imposed on $i$-wise agreements. Note that no addition operation is defined on $\wF_2^k$. 

\begin{theorem} \label{thm:ville}
	Suppose $\ell, k$ are integers with $2 \le \ell \le k$ and consider the vector
    \[
    \valpha_{\ell}^k \eqdef \bigl( \underbrace{*, \dots, *}_{\ell-1}, 1, \underbrace{*, \dots, *}_{k-\ell} \bigr) \in \wF_2^k. 
    \]
    Then the sequence $\bigl( g_{\valpha_{\ell}^k}(n) \bigr)_{n \in \N_+}$ is bounded if and only if $\ell$ is a power of $2$. 
\end{theorem}


\paragraph{Paper organization.} The proof of \Cref{thm:main} is divided among \Cref{sec:la_pre,sec:lower,sec:upper}. In \Cref{sec:la_pre}, we establish \Cref{lem:canonical}, the cornerstone of canonical decompositions, together with several auxiliary results. The lower and upper bounds are then proved in \Cref{sec:lower} and \Cref{sec:upper}, respectively. In \Cref{sec:proof}, we prove \Cref{thm:ville}. We conclude with some remarks and open problems in \Cref{sec:remark}. 

\section{Preliminaries} \label{sec:la_pre}

In this section, we prove \Cref{lem:canonical} and derive several other useful properties of canonical vectors. 


Henceforth we use the abbreviated notation $a \= b$, which refers to $a = b \bmod 2 \in \F_2$. 

\begin{proof}[Proof of \Cref{lem:canonical}]
    Let $\mGamma$ be the infinite matrix whose rows are $\vgamma^1, \vgamma^2, \dots$. We need to show that the $k \times k$ top-left principal submatrix $\mGamma_k$ of $\mGamma$ is $\F_2$-invertible. Permuting the rows of $\mGamma_k$, we obtain $\widetilde{\mGamma}_k$ with $(\widetilde{\mGamma}_k)_{i, j} = \binom{i - \lceil k/2 \rceil - 1}{j - 1} \bmod 2$. This process is illustrated for $k = 6$ below. 
    
    {\footnotesize
    \[
    \mGamma_6 \= \begin{pmatrix}
        \binom{-1}{0} & \binom{-1}{1} & \binom{-1}{2} & \binom{-1}{3} & \binom{-1}{4} & \binom{-1}{5} \\
        \binom{0}{0} & \binom{0}{1} & \binom{0}{2} & \binom{0}{3} & \binom{0}{4} & \binom{0}{5} \\
        \binom{-2}{0} & \binom{-2}{1} & \binom{-2}{2} & \binom{-2}{3} & \binom{-2}{4} & \binom{-2}{5} \\
        \binom{1}{0} & \binom{1}{1} & \binom{1}{2} & \binom{1}{3} & \binom{1}{4} & \binom{1}{5} \\
        \binom{-3}{0} & \binom{-3}{1} & \binom{-3}{2} & \binom{-3}{3} & \binom{-3}{4} & \binom{-3}{5} \\
        \binom{2}{0} & \binom{2}{1} & \binom{2}{2} & \binom{2}{3} & \binom{2}{4} & \binom{2}{5}
    \end{pmatrix}
    \quad \rightsquigarrow \quad
    \widetilde{\mGamma}_6 \= \begin{pmatrix}
        \binom{-3}{0} & \binom{-3}{1} & \binom{-3}{2} & \binom{-3}{3} & \binom{-3}{4} & \binom{-3}{5} \\
        \binom{-2}{0} & \binom{-2}{1} & \binom{-2}{2} & \binom{-2}{3} & \binom{-2}{4} & \binom{-2}{5} \\
        \binom{-1}{0} & \binom{-1}{1} & \binom{-1}{2} & \binom{-1}{3} & \binom{-1}{4} & \binom{-1}{5} \\
        \binom{0}{0} & \binom{0}{1} & \binom{0}{2} & \binom{0}{3} & \binom{0}{4} & \binom{0}{5} \\
        \binom{1}{0} & \binom{1}{1} & \binom{1}{2} & \binom{1}{3} & \binom{1}{4} & \binom{1}{5} \\
        \binom{2}{0} & \binom{2}{1} & \binom{2}{2} & \binom{2}{3} & \binom{2}{4} & \binom{2}{5}
    \end{pmatrix}
    \]
    }
    
    \noindent So, it suffices to show for all $a \in \Z$ that the matrix $\mM_a \eqdef \Bigl( \binom{i+a}{j-1} \bmod 2 \Bigr)_{i, j \in [k]}$ has full rank over $\F_2$. 
    
	We proceed by induction on $a$. Since $\mM_{-1}$ is lower triangular with diagonal entries $\binom{i-1}{i-1} = 1$, the statement holds for $a = -1$. Suppose that $\mM_a$ is invertible. Fix $j$ and regard a typical entry $\binom{x}{j-1} = \frac{x(x-1)\cdots(x-j+2)}{(j-1)!}$ in the row indexed by $x$ as a polynomial in $x$ of degree at most $k-1$. For instance, the row $\Bigl( \binom{i+a}{j-1} \bmod 2 \Bigr)_{j \in [k]}$ is indexed by $x = i + a$. Applying the Lagrange interpolation formula at $a + 1, \, a + 2, \, \dots, \, a + k$ and evaluating at $x = a$, we deduce that
 	\[
    \binom{a}{j-1} = \sum_{s=1}^{k} \Biggl( \prod_{\substack{1 \le t \le k\\t \ne s}} \frac{0-t}{s-t} \Biggr) \cdot \binom{a+s}{j-1} = \sum_{s=1}^{k} (-1)^{s-1} \binom{k}{s} \cdot \binom{a+s}{j-1}. 
    \]
 	Reducing modulo $2$, the coefficient of $\binom{a+k}{j-1}$ equals $1$ in $\F_2$, and so $\Bigl( \binom{a}{0}, \dots, \binom{a}{k-1} \Bigr)$, the row indexed by $x = a$, is a linear combination in $\F_2$ of the rows indexed by $x = a + 1, \, x = a + 2, \, \dots, \, x = a + k$ whose coefficient before the row indexed by $x = a + k$ is $1$. Consequently, replacing the row indexed by $x = a + k$ in matrix $\mM_a$ by the row indexed by $x = a$ is an invertible elementary row operation in $\F_2$. Phrased differently, there exists $\varphi \in \operatorname{SL}_k(\F_2)$ such that $\mM_{a-1} = \varphi \cdot \mM_a$. It follows that the matrix $\mM_{a-1}$ is invertible. To summarize, the above deductions show that
    \vspace{-0.5em}
    \begin{itemize}
        \item if $a \in \Z$ and $\mM_a \in \operatorname{SL}_k(\F_2)$ (i.e., $\mM_a$ is invertible), then $\mM_{a-1} \in \operatorname{SL}_k(\F_2)$. 
    \end{itemize}
    \vspace{-0.5em}

    The same argument as above also shows that
    \vspace{-0.5em}
    \begin{itemize}
        \item if $a \in \Z$ and $\mM_a \in \operatorname{SL}_k(\F_2)$ (i.e., $\mM_a$ is invertible), then $\mM_{a+1} \in \operatorname{SL}_k(\F_2)$. 
    \end{itemize}
    \vspace{-0.5em}
    (Note that we only need $\mM_a$ to be invertible when $a$ is negative for the purpose of proving \Cref{lem:canonical}, and so this is unnecessary.) The inductive proof is complete, and hence $\mGamma_k$ is invertible. 
	
	We thus conclude that $(\vgamma^1)_{[k]}, \dots, (\vgamma^k)_{[k]}$ are linearly independent over $\F_2$. 
\end{proof}

\begin{corollary} \label{lem:top-even-coordinate}
    If $t \in \N_+$ and $x \in \{2t, 2t+1\}$, then vectors $(\vgamma^1)_{[2t]}, \dots, (\vgamma^{2t-1})_{[2t]}, (\vgamma^x)_{[2t]}$ span $\F_2^{2t}$.
\end{corollary}

\begin{proof}
    The case $x = 2t$ is exactly \Cref{lem:canonical}. For the case $x = 2t + 1$, by definition, $(\vgamma^{2t})_j \= \binom{t-1}{j-1}$ and $(\vgamma^{2t+1})_j \= \binom{-(t+1)}{j-1}$. We adopt the notations associated with $\mGamma_{2t}$ in the proof of \Cref{lem:canonical}. The row $(\vgamma^{2t+1})_{[2t]}$ indexed by $x = -(t+1)$ is obtained from the rows $x = t - 1, \, x = t - 2, \, \dots, \, x = -t$ via the same interpolation-based row-replacement, whose leading coefficient is $1$ in $\F_2$. Hence, the linear expansion of $(\vgamma^{2t+1})_{[2t]}$ in the basis $(\vgamma^1)_{[2t]}, \dots, (\vgamma^{2t})_{[2t]}$ contains $(\vgamma^{2t})_{[2t]}$ with coefficient $1$. It follows that the vectors $(\vgamma^1)_{[2t]}, \dots, (\vgamma^{2t-1})_{[2t]}, (\vgamma^{2t+1})_{[2t]}$ span $\F_2^{2t}$ as well.
\end{proof}

For any $\vgamma = (\gamma_1, \gamma_2, \dots) \in \F_2^{\otimes \omega}$, recall that the shifted vector $\sigma(\vgamma) \eqdef (\gamma_2, \gamma_3, \dots)$. The following lemma will be useful in the upper bound proofs of \Cref{thm:main}. 

\begin{lemma} \label{lem:shift}
	For every positive integer $i$, we have the identities of canonical vectors
	\[
	\sigma(\vgamma^{2i-1}) = \vgamma^1 + \vgamma^3 + \dots + \vgamma^{2i-1},	\qquad \sigma(\vgamma^{2i}) = \vgamma^2 + \vgamma^4 + \dots + \vgamma^{2i-2}. 
	\]
    In particular, the shifted vector $\sigma(\vgamma^k)$ lies in the $\F_2$-span of canonical vectors $\vgamma^1, \dots, \vgamma^k$ for $k \ge 1$. 
\end{lemma}

\begin{proof}
	Notice that $\sigma(\vgamma^2) = \bm{0}$ is consistent with the identity. Due to standard binomial identities, 
    \begin{align*}
        \sigma(\vgamma^{2i-1})_j = (\vgamma^{2i-1})_{j+1} \= \binom{i+j-1}{j} &= \sum_{\ell=1}^{i} \binom{\ell+j-2}{j-1} \= \sum_{\ell=1}^i (\vgamma^{2\ell-1})_j, \\
        \sigma(\vgamma^{2i})_j = (\vgamma^{2i})_{j+1} = \binom{i-1}{j} &= \sum_{\ell=1}^{i-1} \binom{\ell-1}{j-1} = \sum_{\ell=1}^{i-1} (\vgamma^{2\ell})_j. \qedhere
    \end{align*}
\end{proof}


In addition to the algebraic properties of the canonical vectors $\vgamma^1, \vgamma^2, \dots$, we will also need a combinatorial tool, known as the \emph{Partition--Sum Lemma}~\cite[Lemma~3]{johnston_oneill}, to establish \Cref{thm:main}. For the sake of self-containment, we also include the proof of Partition--Sum Lemma below. 

\begin{lemma}[Partition--Sum Lemma] \label{lem:Partition--Sum}
	Let $\valpha, \vbeta \in \F_2^k$ and $x, y \in \N_+$. If $\min \bigl\{ f_{\valpha}(x), f_{\vbeta}(y) \bigr\} > k$, then
	\[
    f_{\valpha+\vbeta}(x+y) \ge \min \bigl\{ f_{\valpha}(x), f_{\vbeta}(y) \bigr\}. 
    \]
\end{lemma}

\begin{proof}
	Suppose vectors $\valpha = (\alpha_1, \dots, \alpha_k), \, \vbeta = (\beta_1, \dots, \beta_k)$ and define $\vgamma \eqdef \valpha + \vbeta$. Write $a \eqdef f_{\valpha}(x)$ and $b \eqdef f_{\vbeta}(y)$. According to the definitions, there exist
    \vspace{-0.5em}
    \begin{itemize}
        \item an $\valpha$-town $\cA = \{A_1, \dots, A_a\}$ on ground set $[x] = \{1, 2, \dots, x\}$ of size $a$, and
        \vspace{-0.5em}
        \item a $\vbeta$-town $\cB = \{B_1, \dots, B_b\}$ on ground set $\{x+1, x+2, \dots, x+y\}$ of size $b$. 
    \end{itemize}
    \vspace{-0.5em}
    We claim that the family $\cF \eqdef \bigl\{ A_i \cup B_i : i = 1, 2, \dots, \min\{a, b\} \bigr\}$ is a $\vgamma$-town on ground set $[x+y]$ of size $\min\{a, b\}$, implying that $f_{\vgamma}(x+y) \ge \min \bigl\{ f_{\valpha}(x), f_{\vbeta}(y) \bigr\}$. Since the ground sets of $\cA$ and $\cB$ are disjoint, $A_i \cup B_i$ are pairwise distinct, and so $|\cF| = \min\{a, b\}$. For any $\ell \in [k]$ and any $I \in \binom{\min\{a, b\}}{\ell}$, 
    \[
    \biggl| \bigcap_{i \in I} (A_i \cup B_i) \biggr| = \biggl| \bigcap_{i \in I} A_i \biggr| + \biggl| \bigcap_{i \in I} B_i \biggr| \= \valpha_{\ell} + \vbeta_{\ell} \= \vgamma_{\ell}, 
    \]
    meaning that $\cF$ is a $\vgamma$-town. So, the claim is valid, and the proof is complete. 
\end{proof}

This Partition--Sum Lemma has the following direct corollary. 

\begin{corollary} \label{coro:Partition--Sum}
    Let $\valpha_1, \dots, \valpha_{\ell} \in \F_2^k$ and $x_1, \dots, x_{\ell} \in \N_+$. If $\min \bigl\{ f_{\valpha_1}(x_1), \dots, f_{\valpha_{\ell}}(x_{\ell}) \bigr\} > k$, then
    \[
    f_{\valpha_1+\dots+\valpha_{\ell}}(x_1 + \dots + x_{\ell}) \ge \min \bigl\{ f_{\valpha_1}(x_1), \dots, f_{\valpha_{\ell}}(x_{\ell}) \bigr\}. 
    \]
\end{corollary}

It is worth mentioning that the Partition--Sum Lemma is useful in two different ways. 
\vspace{-0.5em}
\begin{itemize}
    \item On one hand, given constructions of large $\valpha_i$-towns for $i = 1, \dots, \ell$, the lemma yields a large $(\valpha_1 + \dots + \valpha_{\ell})$-town, which will be useful for establishing the lower bounds in \Cref{thm:main}. 
    \vspace{-0.5em}
    \item On the other hand, the lemma acts as a strong triangle inequality, allowing us to combine vectors $\valpha$ and $\vbeta$ from different levels while preserving a smaller upper bound on $f_{\valpha}(n)$. This turns out to be a crucial step in our reductive proof of the upper bounds in \Cref{thm:main}. 
\end{itemize}
\vspace{-0.5em}
These features make the Partition--Sum Lemma our main technical tool in the proof of \Cref{thm:main}. 

\section{The lower bounds in \texorpdfstring{\Cref{thm:main}}{Theorem 3}} \label{sec:lower}

In this section, we deduce the lower bounds restated below in \Cref{thm:main}. 

\begin{theorem} \label{thm:main_lower}
    Suppose $k \in \N_+$ and let $\valpha \in \F_2^k$ with $\lv(\valpha) = t > 0$. Then, as $n \to \infty$, we have
    \[
    f_{\valpha}(n) \ge \begin{cases}
        \bigl( 1 + o(1) \bigr) \cdot (t!)^{1/t} \cdot n^{1/t} \qquad &\text{if $\grd(\valpha) = 1$}, \\
        \bigl( 1 + o(1) \bigr) \cdot (t!/2)^{1/t} \cdot n^{1/t} \qquad &\text{if $\grd(\valpha) = 2$}. 
    \end{cases}
    \]
\end{theorem}

We first construct \emph{canonical families} $\cG_i$, which are asymptotically largest $\vgamma^i$-towns for canonical vectors $\vgamma^i$. We then apply Partition--Sum Lemma to prove \Cref{thm:main_lower} in full generality. It is worth mentioning that throughout this section we always implicitly assume that $t = \lfloor i/2 \rfloor$. 

\smallskip
For $m \in \N_+$ and $t \in [m]$, we denote by $\binom{[m]}{t}$ the family of $t$-element subsets of $[m]$. 
\vspace{-0.5em}
\begin{itemize}
    \item For $s = 1, \dots, m$, we write $A^{(t)}_s \eqdef \Bigl\{ X \in \tbinom{[m]}{t} : s \in X \Bigr\}$. 
\end{itemize}
\vspace{-0.5em}
In other words, the \emph{star} $A^{(t)}_s$ consists of all $t$-element subsets of $[m]$ containing $s$. 

Inspired by the previous lower-bound constructions of Johnston--O'Neill \cite[Construction 3]{johnston_oneill} and Wei--Zhang--Ge \cite[Construction II.1]{wei_zhang_ge}, we introduce set families
\[
\cF_{2t}(m) \eqdef \Bigl\{ A^{(t)}_s : s \in [m] \Bigr\}, \qquad \cF_{2t+1}(m) \eqdef \Bigl\{ \tbinom{[m]}{t} \setminus A^{(t)}_s : s \in [m] \Bigr\}. 
\]
As previously remarked in \Cref{sec:intro}, there exists a $\vgamma^1$-town $\cF_1(m)$ of size $2^{\lfloor \frac{m-1}{2} \rfloor}$ for every $m$. 

\begin{proposition} \label{prop:realization}
    Suppose $t = \lfloor i/2 \rfloor \in \N_+$. If $m$ is divisible by $2^t$, then $\cF_i(m)$ is a $\vgamma^i$-town. 
\end{proposition}

\begin{proof}
    For $j = 1, \dots, t$, recall that $(\vgamma^{2t})_j \= \binom{t-1}{j-1}$ and $(\vgamma^{2t+1})_j \= \binom{-(t+1)}{j-1}$. By constructions, 
    \vspace{-0.5em}
    \begin{itemize}
        \item every $j$-wise intersection in $\cF_{2t}(m)$ has size $\binom{m-j}{t-j}$, and
        \vspace{-0.5em}
        \item every $j$-wise intersection in $\cF_{2t+1}(m)$ has size $\binom{m-j}{t}$. 
    \end{itemize}
    \vspace{-0.5em}
    Denote by $\nu_2(m)$ the largest integer $\lambda$ such that $2^{\lambda}$ divides $m$. Since $m$ is divisible by $2^t$, 
    \[
    \nu_2(m) \ge t \implies \nu_2(m) > \nu_2(t!) \implies \begin{cases}
        \binom{m-j}{t-j} \= \binom{-j}{t-j} = \binom{t-1}{j-1}, \\
        \binom{m-j}{t} \= \binom{-j}{t} = \binom{-(t+1)}{j-1}. 
    \end{cases} \qedhere
    \]
\end{proof}

\begin{proof}[Proof of \Cref{thm:main_lower}]
    We first construct $\cG_i$ using $\cF_i$. When $i = 1$, the families $\cG_1(n) \eqdef \cF_1(n)$ work. Suppose $i \ge 2$ then. Observe that $|\cF_i(m)| = m$ and $\cF_i(m)$ has ground set $\binom{[m]}{t}$. To obtain $\cG_i(n)$, we pick $m_{n, i}$ to be the largest integer with $\binom{m}{t} \le n$ and $\nu_2(m) \ge t$. It follows that
    \[
    m_{n, i} = \bigl( 1 + o(1) \bigr) \cdot (t!)^{1/t} \cdot n^{1/t}, 
    \]
    and so (by \Cref{prop:realization}) the $\vgamma^i$-town $\cG_i(n) \eqdef \cF_i(m_{n, i})$ implies $f_{\vgamma^i}(n) \ge \bigl( 1 + o(1) \bigr) \cdot (t!)^{1/t} \cdot n^{1/t}$. 

    In the general case, suppose $i_1 < \dots < i_{\ell}$ are indices such that
    \[
    \valpha = \vgamma^{i_1} + \dots + \vgamma^{i_{\ell}}
    \]
    is the canonical decomposition. Since $\lv(\valpha) = t$, we have that $i_{\ell} \in \{2t, 2t+1\}$. Moreover, 
    \vspace{-0.5em}
    \begin{itemize}
        \item if $\grd(\valpha) = 1$, then $i_{\ell-1} \le 2t - 1$ if it exists; 
        \vspace{-0.5em}
        \item if $\grd(\valpha) = 2$, then $i_{\ell-1} = 2t$ and $i_{\ell} = 2t + 1$. 
    \end{itemize}
    \vspace{-0.5em}
    Choose parameters $(x_1, \dots, x_{\ell})$ with $f_{\vgamma^{i_1}}(x_1) = \dots = f_{\vgamma^{i_{\ell}}}(x_{\ell}) \eqdef m$ such that $x_1 + \dots + x_{\ell} \le n$ is as large as possible. Then from the Partition--Sum Lemma (\Cref{coro:Partition--Sum}) we deduce that
    \[
    f_{\valpha}(n) \ge f_{\valpha}(x_1 + \dots + x_{\ell}) \ge m = \begin{cases}
        \bigl( 1 + o(1) \bigr) \cdot (t!)^{1/t} \cdot n^{1/t} \qquad &\text{if $\grd(\valpha) = 1$}, \\
        \bigl( 1 + o(1) \bigr) \cdot (t!/2)^{1/t} \cdot n^{1/t} \qquad &\text{if $\grd(\valpha) = 2$}, 
    \end{cases}
    \]
    where the ``$=$'' follows from an optimization on $m$. The proof of \Cref{thm:main_lower} is complete. 
\end{proof}

\section{The upper bounds in \texorpdfstring{\Cref{thm:main}}{Theorem 3}} \label{sec:upper}

In this section, we deduce the upper bounds restated below in \Cref{thm:main}. 

\begin{theorem} \label{thm:main_upper}
    Suppose $k \in \N_+$ and let $\valpha \in \F_2^k$ with $\lv(\valpha) = t > 0$. Then, as $n \to \infty$, we have
    \[
    f_{\valpha}(n) \le \begin{cases}
        \bigl( 1 + o(1) \bigr) \cdot (t!)^{1/t} \cdot n^{1/t} \qquad &\text{if $\grd(\valpha) = 1$}, \\
        \bigl( 1 + o(1) \bigr) \cdot (t!/2)^{1/t} \cdot n^{1/t} \qquad &\text{if $\grd(\valpha) = 2$}. 
    \end{cases}
    \]
\end{theorem}

To prove \Cref{thm:main_upper}, the following replacement lemma is crucial. 

\begin{lemma} \label{lem:replacement}
    Suppose $k \in \N_+$ and let $\valpha, \vbeta \in \F_2^k$ with $\lv(\valpha) = t > s = \lv(\vbeta)$. Then, 
    \vspace{-0.5em}
    \begin{itemize}
        \item given that $f_{\valpha+\vbeta}(n) \le \bigl( c + o(1) \bigr) \cdot n^{1/t}$ as $n \to \infty$, we have $f_{\valpha}(n) \le \bigl( c + o(1) \bigr) \cdot n^{1/t}$ as $n \to \infty$. 
    \end{itemize}
    \vspace{-0.5em}
\end{lemma}

\begin{proof}
    Arbitrarily pick a parameter $\delta \in (0, 1)$ and set $m \eqdef n^{(t-\delta)/t}$. For $s > 0$ and $n \to \infty$, we have
    \[
    m = o(n) \qquad \text{and} \qquad m^{1/s} \ge m^{1/(t-1)} = \omega (n^{1/t}). 
    \]
    It then follows from the Partition--Sum Lemma applied to $\valpha + \vbeta$ and $n + m$ that
    \[
    \bigl( c + o(1) \bigr) \cdot n^{1/t} = \bigl( c + o(1) \bigr) \cdot (n+m)^{1/t} \ge f_{\valpha+\vbeta}(n+m) \ge \min \bigl\{ f_{\valpha}(n), f_{\vbeta}(m) \bigr\}.
    \]
    
    \newpage
    We know from \Cref{thm:main_lower} that $f_{\vbeta}(m) = \Omega(m^{1/s})$, and hence $f_{\vbeta}(m) = \omega(n^{1/t})$\footnote{If $s = 0$, then $m^{1/s}$ is not well-defined. Nevertheless, in this case we still have $f_{\vbeta}(m) = \Omega(2^{m/2}) = \omega(n^{1/t})$.}. Therefore, 
    \[
    f_{\valpha}(n) \le \bigl( c + o(1) \bigr) \cdot n^{1/t}. \qedhere
    \]
\end{proof}


\begin{proof}[Proof of \Cref{thm:main_upper}]
    We analyze two cases on $\grd(\valpha)$ separately. 
    
	
	
	\smallskip
	\noindent\textbf{Case 1. $\grd(\valpha) = 1$.} 
    \smallskip

    We first treat the case $\valpha_t = 1$ and $\valpha_{t+1} = \dots = \valpha_{2t} = 0$. This follows from an argument of Wei, Zhang, and Ge \cite[Lemma IV.1]{wei_zhang_ge}. For completeness, we spell out the details below. 
    
	Let $\cF = \{F_1, \dots, F_m\} \subseteq 2^{[n]}$ be an $\valpha$-town. Consider the family of $t$-wise intersections
	\[
    \mathcal G \eqdef \biggl\{ \bigcap_{j \in T} F_j : T \in \tbinom{[m]}{t} \biggr\}. 
    \]
	We claim that all members of $\cG$ are distinct, and so $|\cG| = \binom{m}{t}$. Indeed, if $\bigcap\limits_{j \in T} F_j = \bigcap\limits_{j \in T'} F_j$ for some $T \ne T'$, then intersecting both sides with $F_r \, (r \in T \bigtriangleup T')$ yields an equality between a $(t+1)$-wise intersection and a $t$-wise intersection of distinct sets, contradicting $\valpha_t = 1$ while $\valpha_{t+1} = 0$. 
	
	For distinct $G, G' \in \mathcal G$, their intersection is an $\ell$-wise intersection of distinct members of $\cF$, where $t + 1 \le \ell \le 2t$. It then follows from $\valpha_{t+1} = \dots = \valpha_{2t} = 0$ that $|G \cap G'|$ is even. Consequently, $\cG$ is an oddtown on $[n]$, hence $\tbinom{m}{t} = |\cG| \le n$. This implies that $m \le \bigl( 1 + o(1) \bigr) \cdot (t!n)^{1/t}$. Consequently, 
    \[
    f_{\valpha}(n) \le \bigl( 1 + o(1) \bigr) \cdot (t!)^{1/t} \cdot n^{1/t}. 
    \]

    We then reduce the general case to the special case above. We have canonical decomposition
    \[
    \valpha = c_1 \cdot \vgamma^1_{[k]} + \dots + c_{2t-1} \cdot \vgamma^{2t-1}_{[k]} + 1 \cdot \vgamma^x_{[k]}, 
    \]
    where $x \in \{2t, 2t+1\}$. Introduce an auxiliary vector $\valpha^* \in \F_2^k$ such that
    \[
    \valpha^*_{[2t]} \eqdef \bigl( \underbrace{1, \dots, 1}_t, \underbrace{0, \dots, 0}_{t} \bigr). 
    \]
    We remark that only the last $t + 1$ coordinates $\valpha^*_t = 1$ and $\valpha^*_{t+1} = \dots = \valpha^*_{2t} = 0$ matter. Due to \Cref{lem:top-even-coordinate}, the vectors $\vgamma^1_{[2t]}, \dots, \vgamma^{2t-1}_{[2t]}, \vgamma^x_{[2t]}$ span $\F_2^{2t}$. Hence, we have linear decomposition
    \[
    \valpha_{[2t]} + \valpha^*_{[2t]} = \lambda_1 \cdot \vgamma^1_{[2t]} + \dots + \lambda_{2t-1} \cdot \vgamma^{2t-1}_{[2t]} + \lambda_x \cdot \vgamma^x_{[2t]}, 
    \]
    where $\lambda_1, \dots, \lambda_{2t-1}, \lambda_x \in \F_2$ are uniquely determined. Complete the construction of $\valpha^*$ by setting
    \[
    \valpha^* \eqdef \valpha + \lambda_1 \cdot \vgamma^1_{[k]} + \dots + \lambda_{2t-1} \cdot \vgamma^{2t-1}_{[k]} + \lambda_x \cdot \vgamma^x_{[k]}. 
    \]
    
    A crucial observation is that $\lambda_x = 0$. If not, then $\lambda_x = 1$, and so $\valpha^*_{[2t]} \in \spa_{\F_2} \bigl\{ \vgamma^1_{[2t]}, \dots, \vgamma^{2t-1}_{[2t]} \bigr\}$. This implies that $\lv(\valpha^*_{[2t]}) \le t - 1$, and hence we obtain $f_{\valpha^*_{[2t]}} = \Omega(n^{1/(t-1)})$, thanks to \Cref{thm:main_lower}. However, it follows from the coordinates $\valpha^*_t = 1$ and $\valpha^*_{t+1} = \dots = \valpha^*_{2t} = 0$ that $f_{\valpha^*_{[2t]}} = O(n^{1/t})$, a contradiction. Therefore, we see that $\lambda_x = 0$, and hence $\valpha^* = \valpha + \lambda_1 \cdot \vgamma^1_{[k]} + \dots + \lambda_{2t-1} \cdot \vgamma^{2t-1}_{[k]}$. 

    Finally, by \Cref{lem:replacement} applied to $\vbeta \eqdef \lambda_1 \cdot \vgamma^1_{[k]} + \dots + \lambda_{2t-1} \cdot \vgamma^{2t-1}_{[k]}$, we conclude that
    \[
    f_{\valpha}(n) \le \bigl( 1 + o(1) \bigr) \cdot (t!)^{1/t} \cdot n^{1/t}. 
    \]
    
	\smallskip
	\noindent\textbf{Case 2. $\grd(\valpha) = 2$.} 
    \smallskip
    
	We first claim that there exists $\vbeta \in \spa_{\F_2} \bigl\{ \vgamma^1_{[k]}, \dots, \vgamma^{2t-1}_{[k]} \bigr\}$ such that $\valpha' \eqdef \valpha + \vbeta$ satisfies
    \[
    \valpha'_1 = \valpha'_2 = 0, \qquad \valpha'_3 = 1. 
    \]
    We verify this by considering the following cases for $t$. 
    \vspace{-0.5em}
    \begin{itemize}
        \item If $t \ge 2$, then \Cref{lem:canonical} shows that $\vgamma^1_{[3]}, \vgamma^2_{[3]}, \vgamma^3_{[3]}$ span $\F_2^3$. So, there exist $\mu_1, \mu_2, \mu_3 \in \F_2$ with
        \[
        (0, 0, 1) - \valpha_{[3]} = \mu_1 \cdot \vgamma^1_{[3]} + \mu_2 \cdot \vgamma^2_{[3]} + \mu_3 \cdot \vgamma^3_{[3]}, 
        \]
        and hence $\vbeta \eqdef \mu_1 \cdot \vgamma^1_{[k]} + \mu_2 \cdot \vgamma^2_{[k]} + \mu_3 \cdot \vgamma^3_{[k]}$ satisfies $(\valpha + \vbeta)_{[3]} = (0, 0, 1) \in \F_2^3$. 
        \vspace{-0.5em}
        \item If $t = 1$, then there exists $\mu \in \F_2$ such that
        \[
        \valpha_{[3]} = \mu \cdot \vgamma^1_{[3]} + \vgamma^2_{[3]} + \vgamma^3_{[3]} = \mu \cdot (1, 1, 1) + (0, 0, 1),
        \]
        and so $\vbeta \eqdef \mu \cdot \vgamma^1_{[k]}$ satisfies $(\valpha + \vbeta)_{[3]} = (0, 0, 1) \in \F_2^3$. 
    \end{itemize}
    \vspace{-0.5em}

    Thanks to \Cref{lem:replacement}, it then suffices to show that $f_{\valpha'}(n) \le \bigl( 1 + o(1) \bigr) \cdot (t!/2)^{1/t} \cdot n^{1/t}$. 

    We extend the definition of shifting operator $\sigma$ by denoting $\sigma(\vgamma) \eqdef (\vgamma_2, \dots, \vgamma_{k}) \in \F_2^{k-1}$ for each $\vgamma \in \F_2^k$. Crucially, we may deduce via \Cref{lem:shift} that both $\sigma(\valpha')$ and $\valpha'_{[k-1]} + \sigma(\valpha')$ have level $t$ and grade $1$. Indeed, the shifting $\sigma$ and truncating $(\bullet)_{[k-1]}$ are linear operators with the property that
    \vspace{-0.5em}
    \begin{itemize}
        \item the identity $\sigma(\vgamma_{[k]}) = \sigma(\vgamma)_{[k-1]}$ holds for every $\vgamma \in \F_2^{\otimes\omega}$. 
    \end{itemize}
    \vspace{-0.5em}
    Since $\valpha'$ has level $t$ and grade $2$, by introducing unspecified $\bm{*}_{\ell} \in \spa_{\F_2} \bigl\{ \vgamma^1_{[\ell]}, \dots, \vgamma^{2t-1}_{[\ell]} \bigr\}$ we compute
    \begin{align*}
        \valpha' = \bm{*}_k + 1 \cdot \vgamma^{2t}_{[k]} + 1 \cdot \vgamma^{2t+1}_{[k]} \implies \sigma(\valpha') &= \sigma(\bm{*}_k) + \sigma(\vgamma^{2t}_{[k]}) + \sigma(\vgamma^{2t+1}_{[k]}) \\
        &= \sigma(\bm{*}_k) + \sigma(\vgamma^{2t})_{[k-1]} + \sigma(\vgamma^{2t+1})_{[k-1]} \\
        &= \bm{*}_{k-1} + \bm{*}_{k-1} + \bigl( \bm{*}_{k-1} + \vgamma^{2t+1}_{[k-1]} \bigr) \qquad \text{by \Cref{lem:shift}} \\
        &= \bm{*}_{k-1} + 0 \cdot \vgamma^{2t}_{[k-1]} + 1 \cdot \vgamma^{2t+1}_{[k-1]}. 
    \end{align*}
    Since $\valpha'_{[k-1]} = \bm{*}_{k-1} + 1 \cdot \vgamma^{2t}_{[k-1]} + 1 \cdot \vgamma^{2t+1}_{[k-1]}$, we obtain $\valpha'_{[k-1]} + \sigma(\valpha') = \bm{*}_{k-1} + 1 \cdot \vgamma^{2t}_{[k-1]} + 0 \cdot \vgamma^{2t+1}_{[k-1]}$. 

    \smallskip
    Let $\cG = \{F_0, F_1, \dots, F_m\} \subseteq 2^{[n]}$ be an $\valpha'$-town with $\cF \eqdef \{F_1, \dots, F_m\}$. Write $F_0^{\sc} \eqdef [n] \setminus F_0$ and
    \[
    \cG_1 \eqdef F_0 \cap \cF = \bigl\{ F_0 \cap F : F \in \cF \bigr\}, \qquad \cG_2 \eqdef F_0^{\sc} \cap \cF = \bigl\{ F_0^{\sc} \cap F : F \in \cF \bigr\}. 
    \]
    By construction we have that $\cG_1$ is a $\sigma(\valpha')$-town on $F_0$, and that $\cG_2$ is an $\bigl( \valpha'_{[k-1]} + \sigma(\valpha') \bigr)$-town on $F_0^{\sc}$. The latter is implied by $(F_0^{\sc} \cap F_{i_1}) \cap \dots \cap (F_0^{\sc} \cap F_{i_j}) = (F_{i_1} \cap \dots \cap F_{i_j}) \setminus (F_0 \cap F_{i_1} \cap \dots \cap F_{i_j})$. 

    Since $\valpha'_{[3]} = (0, 0, 1)$, we have $\sigma(\valpha')_1 \ne \sigma(\valpha')_2$ and $\bigl( \valpha'_{[k-1]} + \sigma(\valpha') \bigr)_1 \ne \bigl( \valpha'_{[k-1]} + \sigma(\valpha') \bigr)_2$. Thus, a similar argument as in \textbf{Case 1} shows that all members of $\cG_1$ and all members of $\cG_2$ are distinct. Because both $\sigma(\valpha')$ and $\valpha'_{[k-1]} + \sigma(\valpha')$ have level $t$ and grade $1$, the conclusion of \textbf{Case 1} implies
    \[
    m = |\cG_1| \le \bigl( 1 + o(1) \bigr) \cdot (t!)^{1/t} \cdot |F_0|^{1/t}, \qquad m = |\cG_2| \le \bigl( 1 + o(1) \bigr) \cdot (t!)^{1/t} \cdot |F_0^{\sc}|^{1/t}. 
    \]
    From $\min \bigl\{ |F_0|, |F_0^{\sc}| \bigr\} \le n/2$ we obtain $|\cG| = m + 1 \le \bigl( 1 + o(1) \bigr) \cdot (t!)^{1/t} \cdot (n/2)^{1/t}$. Consequently, 
    \[
    f_{\valpha'}(n) \le \bigl( 1 + o(1) \bigr) \cdot (t!/2)^{1/t} \cdot n^{1/t}. 
    \]
    
	\smallskip
    Combining \textbf{Case 1} and \textbf{Case 2} above completes the proof of \Cref{thm:main_upper}. 
\end{proof}

\section{Proof of \texorpdfstring{\Cref{thm:ville}}{Theorem 4}} \label{sec:proof}

Recall that $\valpha_{\ell}^k = \bigl( \underbrace{*, \dots, *}_{\ell-1}, 1, \underbrace{*, \dots, *}_{k-\ell} \bigr) \in \wF_2^k$. We prove two directions separately. 

\paragraph{Proof that \texorpdfstring{$g_{\valpha_{\ell}^k}(n)$}{g(n)} is unbounded if \texorpdfstring{$\ell$}{l} is not a power of \texorpdfstring{$2$}{2}.} There exists $t \in [\ell-1]$ such that $\binom{-\ell}{t} \= 1$. Choose $m$ to be a sufficiently large integer which is divisible by $2^t$.

Consider the canonical construction $\cF \eqdef \cF_{2t}(m)$ from \Cref{sec:lower}. That is, 
\[
A^{(t)}_s \eqdef \Bigl\{ X \in \tbinom{[m]}{t} : s \in X \Bigr\}, \qquad \cF_{2t}(m) \eqdef \Bigl\{ A^{(t)}_s : s \in [m] \Bigr\}. 
\]
For any distinct $F_1, \dots, F_{\ell} \in \cF$, from $t < \ell$ and $\nu_2(m) \ge t > \nu_2(t!)$ we deduce that
\[
|F_1 \cap \dots \cap F_{\ell}| = 0, \qquad |F_1 \cup \dots \cup F_{\ell}| = \tbinom{m}{t} - \tbinom{m-\ell}{t} \= \tbinom{-\ell}{t} \= 1. 
\]
It follows that $|F_1 \cap \dots \cap F_{\ell}| + |F_1 \cup \dots \cup F_{\ell}| \= 0 + 1 \= 1$. Moreover, we have $|\cF_{2t}(m)| = m$. 

Upon setting $m$ to be the largest integer such that $\binom{m}{\ell} \le n$, we obtain $g_{\valpha_{\ell}^k}(n) = \Omega(n^{1/\ell})$. 

\paragraph{Proof that \texorpdfstring{$g_{\valpha_{\ell}^k}(n)$}{g(n)} is bounded if \texorpdfstring{$\ell$}{l} is a power of \texorpdfstring{$2$}{2}.} Assume $\ell = 2^t$, where $t \in \N_+$. 

Let $\cF = \{F_1, \dots, F_m\}$ be an $\valpha_{\ell}^k$-ville. Take a coloring $\chi \colon \binom{\cF}{\ell} \to \{0, 1\}^D$, where $D \eqdef 2^\ell - 1$ and
\vspace{-0.5em}
\begin{itemize}
    \item each $\ell$-subset $\{F_{i_1}, \dots, F_{i_\ell}\} \subseteq \cF$ with $i_1 < \cdots < i_\ell$ is assigned a $\{0, 1\}$-vector of length $2^\ell - 1$ that records, in a fixed order, the parities of all nonempty intersections among $F_{i_1}, \dots, F_{i_\ell}$.
\end{itemize}
\vspace{-0.5em}
Thanks to Ramsey's theorem, there exists some sufficiently large $M_\ell > 0$ such that if $m > M_\ell$ then $\cF$ contains a subfamily $\cG = \{G_1, \dots, G_{10\ell^2}\}$ in which every $\ell$-subset of $\cG$ receives the same color $\tau$. 

We are going to bound $|\cF|$ by $M_{\ell}$, implying that $g_{\valpha_{\ell}^k}(n) \le M_\ell$. 

We first show that $\tau$ is \emph{symmetric}: 
\vspace{-0.5em}
\begin{itemize}
    \item for $s = 1, \dots, \ell$, all coordinates of $\tau$ corresponding to $s$-wise intersections are equal. 
\end{itemize}
\vspace{-0.5em}
Fix $s \in [\ell]$ and consider the distinguished $s$-tuple $(H_1, H_2, \dots, H_s) \eqdef (G_\ell, G_{2\ell}, \dots, G_{s\ell})$ which exists since $|\cG| = 10\ell^2 \ge s\ell$. We verify that $(H_1, \dots, H_s)$ can occupy any positions within an $\ell$-subset of $\cG$. This implies the desired symmetry, since every $\ell$-subset of $\cG$ receives the same color $\tau$. For any $I = \{i_1, \dots, i_s\} \subseteq [\ell]$ with $i_1 < \dots < i_s$, define an $\ell$-subset $\cG_I \subseteq \cG$ as follows: 
\vspace{-0.5em}
\begin{itemize}
    \item arrange its elements in increasing order of indices, place $H_r$ in the $i_r$-th position, and fill the other $\ell - s$ positions with arbitrary distinct elements of $\{G_1, \dots, G_{(s+1)\ell}\} \setminus \{H_1, \dots, H_s\}$. 
\end{itemize}
\vspace{-0.5em}
This is possible because there are $\ell - 1$ available fillers between any consecutive $H_i, H_{i+1}$, as well as before $H_1$ and after $H_s$. Since each $\ell$-subset of $\cG$ receives color $\tau$, it follows that $\cG_I$ has color $\tau$, and so the $s$-wise coordinate of $\tau$ indexed by $I$ equals $\bigl| \bigcap_{r=1}^s H_r \bigr| \bmod 2$, which is independent of $I$. As $I$ ranges over all $\binom{\ell}{s}$ choices, all $s$-wise coordinates of $\tau$ coincide, proving the desired symmetry.

We then show that $m \le M_\ell$. Fix any distinct $X_1, \dots, X_\ell \in \cG$ and write for $i = 1, \dots, \ell$
\[
x_i \eqdef \bigl| (X_1 \cap \dots \cap  X_i) \setminus (X_{i+1} \cup \dots \cup X_\ell) \bigr|, 
\]
which are the sizes of Venn-diagram cells grouped by depth. The symmetry of $\tau$ yields
\[
|X_1 \cap \dots \cap X_\ell| + |X_1 \cup \dots \cup X_\ell| \overset{\F_2}{=} x_\ell + \Bigl( \tbinom{\ell}{0} \cdot x_\ell + \tbinom{\ell}{1} \cdot x_{\ell-1} + \dots + \tbinom{\ell}{\ell-1} \cdot x_1 \Bigr). 
\]
Since $\ell$ is a power of $2$, the binomial coefficients $\binom{\ell}{1}, \dots, \binom{\ell}{\ell-1}$ are all even, and so the right-hand side is $0$ in $\F_2$, contradicting the assumed oddness of $\valpha_{\ell}^k$. We thus conclude that $m \le M_\ell$. 

\section{Concluding remarks} \label{sec:remark}

We now introduce a general framework encompassing towns and villes. Let $m \ge 2$ and $k \ge 2$ be integers, and let $\vI = (I_1, \dots, I_k) \in \bigl( 2^{\Z/m\Z} \bigr)^k$ be a sequence of subsets of $\Z/m\Z$. A family $\cF \subseteq 2^{[n]}$ is called an $\vI$-town (resp.~$\vI$-ville) if, for each $j = 1, \dots, k$, every $j$-wise intersection (resp.~intersection-union) of distinct sets in $\cF$ has cardinality modulo $m$ lying in $I_j$. Denote by $f_{m, \vI}(n)$ (resp.~$g_{m, \vI}(n)$) the maximum size of an $\vI$-town (resp.~$\vI$-ville) on $[n]$. The ultimate goal is to determine the sharp asymptotic behavior of $f_{m, \vI}(n)$ and $g_{m, \vI}(n)$ as $n \to \infty$, for every $m$ and every $\vI$. 

\smallskip

Our main result, \Cref{thm:main}, resolves the $f_{2,\vI}(n)$ problem in the case where each $I_j$ is a singleton. However, the asymptotic behavior of $f_{2, \vI}(n)$ for general $\vI$ (equivalently, $f_{\valpha}(n)$ for $\valpha \in \wF_2^k$) remains much less understood and appears to merit further investigation. 

\begin{problem}
    Study the sharp asymptotic behavior of $f_\valpha(n)$ for every $k \in \N_+$ and every $\valpha \in \wF_2^k$. 
\end{problem}

As for the function $g$, we have studied the problem of determining $g_{2,\vI}(n)$ in the case where all but one of the $I_j$ are equal to $\Z/2\Z$, showing that $g$ is unbounded unless the exceptional coordinate is a power of $2$. However, the correct asymptotics of those unbounded $g$ remain unknown. 

\begin{problem}
    Find the sharp asymptotic behavior of $g_{\valpha_\ell^k}(n)$ when $\ell$ is not a power of $2$. 
\end{problem}

For the more general function $f_{m,\vI}(n)$, very little is known. Johnston and O'Neill~\cite{johnston_oneill} deduced
\[
\Omega(n) \le f_{3,(0,0,\{1,2\})}(n) \le O(n^2),
\]
and we reiterate their open problem~\cite[Problem 1]{johnston_oneill} on the correct asymptotics of this quantity. 

\begin{problem}
    Determine the sharp asymptotic behavior of $f_{3, \vI}(n)$ for $\vI = (0, 0, \{1, 2\})$. 
\end{problem}

\section*{Acknowledgments}

This work was initiated at the $3^{\text{rd}}$ ECOPRO Student Research Program in the summer of 2025. LW would like to thank ECOPRO for hosting them. Part of this work was done during a visit of ZD and MO to Shandong University. They are thankful to Guanghui Wang for his support. ZD and LW also thank Dilong Yang for insightful discussions during the early stages of this project. 


\section*{Disclosure}

The authors used Chat GPT Pro for language editing assistance in the preparation of this paper. All mathematical results, proofs, and arguments were developed by the authors. 

\bibliographystyle{abbrv}
\bibliography{odd_even_town}

\end{document}